\documentclass[12pt,a4paper]{amsart}
\usepackage[top=35mm, bottom=35mm, left=30mm, right=30mm]{geometry}
\usepackage{mathptmx}
\usepackage{mathrsfs}
\usepackage{amscd}
\usepackage{verbatim}
\usepackage{url}
\usepackage[all]{xy}
\usepackage{color}
\usepackage[colorlinks=true,citecolor=blue]{hyperref}
\usepackage{tikz}
\usetikzlibrary{arrows}
\usetikzlibrary{patterns}
\usetikzlibrary{lindenmayersystems,backgrounds}

\theoremstyle{plain}

\newtheorem{thm}{Theorem}[section]
\newtheorem{lem}[thm]{Lemma}
\newtheorem{prop}[thm]{Proposition}

\theoremstyle{definition}

\newcommand{\mZ}{\mathbb Z}

\begin{document}
\title{The nonexistence of expansive polycyclic group actions on the circle $\mathbb S^1$}

\author[E.~Shi]{Enhui Shi}
\address[E. Shi]{Soochow University, Suzhou, Jiangsu 215006, China}
\email{ehshi@suda.edu.cn}

\author[S.~Wang]{Suhua Wang}
\address[S.~Wang]{Suzhou Vocational University, Suzhou, Jiangsu 215104, China}
\email{wangsuhuajust@163.com}

\author[Z.~Xie]{Zhiwen Xie}
\address[Z.~Xie]{Soochow University, Suzhou, Jiangsu 215006, China}
\email{20204007002@stu.suda.edu.cn}

\author[H.~Xu]{Hui Xu}
\address[H. Xu]{CAS Wu Wen-Tsun Key Laboratory of Mathematics, University of Science and
Technology of China, Hefei, Anhui 230026, China}
\email{huixu2734@ustc.edu.cn}



\maketitle


\begin{abstract}
We show that the circle $\mathbb S^1$ admits no expansive polycyclic group actions.
\end{abstract}

\pagestyle{myheadings} \markboth{E. Shi, S. Wang, Z. Xie, and H. Xu}{The nonexistence of expansive polycyclic group actions}

\section{Introduction}

 Expansivity  is closely related to the structural stability in differential dynamical systems.
Which spaces can admit an expansive homeomorphism has been intensively studied. It is known that the Cantor set, the $2$-adic solenoid, and the tori $\mathbb T^n$ with $n\geq 2$ admit expansive homeomorphisms.  O'Brien and Reddy showed that every compact orientable surface of positive genus admits an expansive
homeomorphism \cite{OR}. However, the circle  admits no expansive homeomorphisms \cite{JU}. It was proved
by Kato and Mouron that several classes of one-dimensional continua admit no expansive homeomorphisms \cite{Ka90, Ka96, Mo02, Mo09}.
Hiraide obtained the nonexistence of expansive homeomorphisms on the sphere $\mathbb S^2$ (see \cite{Hi}).
 Ma\~n\'e showed that no infinite dimensional compact metric space admits an expansive homeomorphism,
and no compact metric space with positive dimension admits a minimal expansive homeomorphism \cite{Mane}; the latter
result of Ma\~n\'e was extended to the case of pointwise recurrence by Shi, Xu, and Yu very recently \cite{SXY}.

\medskip

T. Ward once asked whether the circle $\mathbb S^1$ can admit an expansive nilpotent group action. This question was
answered in  negative by Connell, Furman and Hurder using the technique of semi-ping-pong in an unpublished paper, which is also implied by
Margulis' work \cite{Mar}. Mai, Shi, and Wang showed the nonexistence of expansive actions by commutative or nilpotent groups on Peano continua
with a free dendrite \cite{MS07, SW}. Recently, Liang, Shi, Xu, and Xie showed that if the group $G$ is of subexponential growth
and $X$ is a Suslinian continuum, then $G$ cannot act on $X$ expansively \cite{LSXX}.  Contrary to the case of $\mathbb Z$ action,
Shi and Zhou constructed an expansive $\mathbb Z^2$ action on an infinite dimensional continuum \cite{SZ}; Meyerovitch and Tsukamoto
constructed a minimal expansive $\mathbb Z^2$ action on a compact metric space $X$ with ${\rm dim}(X)>0$ (see \cite{MT19}). Mouron constructed for each positive integer $n$,
a continuum $X$ which admits an expansive $\mathbb Z^{n+1}$ action but admits no expansive $\mathbb Z^n$ actions \cite{Mo10}. These examples
indicate that there are essential differences between expansive $\mathbb Z$ actions and expansive $\mathbb Z^n$ actions with $n>1$.
One may consult \cite{BL, CL15, LS99} for some interesting studies around expansive $\mathbb Z^n$ actions and consult \cite{SX}
for an application of expansive $\mathbb Z^2$ actions to Ramsey theory.
\medskip

The purpose of the paper is to continue the study of the existence of expansive group actions on continua.
Explicitly, we obtain the following theorem.

\begin{thm}\label{mainthm}
The circle $\mathbb S^1$ admits no expansive polycyclic group actions.
\end{thm}

Here we give some remarks on the condition in the main theorem. It is known that a polycyclic group may
have exponential growth. In fact, Wolf showed that a polycyclic group is either virtually nilpotent or
has exponential growth \cite{Wo}. Thus Theorem \ref{mainthm} is not implied by the main theorem in \cite{LSXX}.
In addition, Rosenblatt proved that every finitely generated solvable group either is virtually nilpotent or
contains a free non-abelian subsemigroup \cite{Ro}. So a non-virtually-nilpotent polycyclic group must contain
a free non-abelian subsemigroup; and thus the argument of using semi-ping-pong technique does not work in this
case. The proof of Theorem \ref{mainthm} relies on a detailed description of the structure of any
expansive subgroup of ${\rm Homeo}([0, 1])$ and the existence of quasi-invariant Radon measure for any polycyclic
subgroup of ${\rm Homeo}(\mathbb R)$ established by Plante \cite{Pl}. At last, we should note that
there does exist an expansive solvable group action on $[0, 1]$ (and so does on $\mathbb S^1$) (see e.g. \cite{SZ}).

\section{Existence of minimal open intervals}

Let us first recall some definitions around group actions.
Given a group $G$ and a topological space $X$. Let ${\rm Homeo}(X)$ be
the homeomorphism group of $X$.  A group homomorphism $\phi: G\rightarrow {\rm Homeo}(X)$ is called a {\it continuous
action} of $G$ on $X$; we use the symbol $(X, G, \phi)$ to denote this action and also call it a {\it dynamical system}.
 For brevity, we usually use $gx$ or $g(x)$ instead of $\phi(g)(x)$
and use $(X, G)$ instead of $(X, G, \phi)$ if no confusion occurs. For
$x\in X$, the set $Gx:=\{gx: g\in G\}$ is call the {\it orbit} of $x$ under the action of $G$;
if $Gx=\{x\}$, then $x$ is called a {\it fixed point} of $G$; if $Gx$ is finite, then $x$ is called
a {\it periodic point} of $G$; if $Gx$ is dense in $X$, then
the action $(X, G)$ is called {\it topologically transitive} and $x$ is called a topologically transitive point of $(X, G)$;
if every point of $X$ is transitive, then $(X, G)$ is called {\it minimal}. A subset $E$ of $X$ is said to be
{\it $G$-invariant} if $Gx\subset E$ for every $x\in E$; thus if $E$ is $G$-invariant,
we naturally get a {\it restriction action} $(E, G|_E)$ of $G$ on $E$; $(E, G|_E)$ is called a
{\it subaction} or a {\it subsystem} of $(X, G)$. It is well known that
if $X$ is a Polish space and $G$ is countable, then $(X, G)$ is topologically transitive
if and only if for every nonempty open sets $U, V$ in $X$, there is some $g\in G$ such that
$gU\cap V\not=\emptyset$; if $X$ is a compact metric space, $(X, G)$ is minimal if and only if
it contains no proper closed subsystem.  If $X$ is a compact metric space with metric
$d$, then the action $(X, G)$ is called {\it expansive} if there is some $c>0$ such that
for every $x\not= y\in X$, there is some $g\in G$ such that $d(gx, gy)>0$; such $c$ is called
an {\it expansivity constant} of $(X, G)$.

\begin{lem}\label{transitive open set}
Let $G$ be a countable group acting continuously on the closed interval $[0, 1]$. If the action is expansive, then there is a $G$-invariant nonempty open set
$U$ in $[0, 1]$ such that the subsystem $(U, G|_U)$ is topologically transitive.
\end{lem}

\begin{proof}
Let $c>0$ be an expansivity constant for the action $([0, 1], G)$.  Take $0=x_1<x_2<\cdots<x_{n-1}<x_n=1$ such that $x_{i+1}-x_i<c$ for each $i$.
Set $A=\cup_{i=1}^n\overline{Gx_i}$. Then $A$ is a $G$-invariant closed subset in $[0, 1]$. If for each $i$, $\overline{Gx_i}$ is nowhere dense,
then $A$ is nowhere dense. Thus $[0, 1]\setminus A$ is nonempty and open. Choose $x\not=y$ in the same connected component of $[0, 1]\setminus A$.
Then $|gx-gy|<c$ for each $g\in G$, which contradicts the expansivity of $([0, 1], G)$. So there exists some $i_0$ such that the
interior $U$ of $\overline{Gx_{i_0}}$ is nonempty. Clearly, $U$ is $G$-invariant and $(U, G|_U)$ is topologically transitive.
\end{proof}

The following lemma is well known and easy to be checked.

\begin{lem}\label{finite index expansive}
Let $X$ be a compact metric space and let $G$ be a group acting continuously on $X$. Suppose $H$ is a finite index subgroup of $G$.
If $(X, G)$ is expansive, then so is the subgroup action $(X, H)$.
\end{lem}

\begin{prop}\label{minimal open interval}
Let $G$ be a group acting continuously on the closed interval $[0, 1]$. If the action is expansive, then there is a subgroup $H$ of $G$,
and an $H$-invariant open interval $(a, b)$ in $[0, 1]$ such that the restriction action $((a, b), H|_{(a, b)})$ is minimal and
$([a, b], H|_{[a, b]})$ is expansive.
\end{prop}

\begin{proof}
Fix an expansivity constant $c>0$ for the action $([0, 1], G)$. By Lemma \ref{transitive open set},
we can take a $G$-invariant nonempty open set $U$ in $[0, 1]$ with $(U, G|_U)$ being topologically transitive.
Let $NT(U)$ be the set of all nontransitive points of $(U, G|_U)$.
\medskip

{\bf Claim A.} $U\setminus \overline{NT(U)}$ is nonempty. Otherwise, we can take $x_1<x_2<\cdots<x_n$ in $NT(U)$
such that  each connected component of $U\setminus\{x_1, x_2, \cdots, x_n\}$ has diameter less than $c$. Since
each $x_i$ is a nontransitive point, the closure $\overline{Gx_i}$ contains no interior point. Thus
$U\setminus \cup_{i=1}^n\overline{Gx_i}$ is nonempty and $G$-invariant. Take a connected component $A$ of $U\setminus \cup_{i=1}^n\overline{Gx_i}$.
Then the diameter ${\rm diam}(gA)<c$ for each $g\in G$. This contradicts the expansivity of $([0, 1], G)$. Thus Claim A holds.
\medskip

From Claim A, we can take a maximal open interval $(a, b)$ in $U\setminus \overline{NT(U)}$. Let $H=\{g\in G: g(a, b)=(a, b)\}$.
\medskip

{\bf Claim B.} $((a, b), H|_{(a, b)})$ is minimal. In fact, from the maximality of $(a, b)$,
we have $g(a, b)\cap (a, b)=\emptyset$ for every $g\in G\setminus H$. This together with the topological transitivity
of $(U, G|_U)$ implies Claim B.
\medskip

{\bf Claim C.} $([a, b], H|_{[a, b]})$ is expansive. This is clear if $H$ has finite index in $G$ by Lemma \ref{finite index expansive}.
So we may assume that the index $[G: H]=\infty$. Let $G=\cup_{i=1}^\infty g_iH$ be the coset decomposition of $G$ with respect to $H$.
Then the sets $g_iU (i=1,2,\cdots)$ are pairwise disjoint. Thus there is some $N>0$ such that ${\rm diam}(g_iU)<c$ for all $i>N$.
By the uniform continuity, there is some $\delta>0$ such that $|g_ix-g_iy|<c$ for $i=1, 2,\cdots, N$, whenever $|x-y|<\delta$.
If $([a, b], H|_{[a, b]})$ is not expansive, then there are $x'\not=y'\in [a, b]$ with $|gx'-gy'|<\delta$ for each $g\in H$.
Thus $|gx'-gy'|\leq c$ for all $g\in G$. This  contradicts the expansivity of $G$.
\medskip

We complete the proof from Claim B and Claim C.
\end{proof}

\section{Minimal polycyclic subgroups of ${\rm Homeo}(\mathbb R)$}

Recall that a group $G$ is {\it polycyclic} if it has a decreasing series of normal subgroups
$G=N_0\rhd N_1\rhd \cdots\rhd N_n\rhd N_{n+1}=\{e\}$ such that $N_i/N_{i+1}$ is cyclic for each $i\geq 0$.
It is known that every finitely generated nilpotent group is polycyclic and every polycyclic group is solvable;
every subgroup and every quotient group of a polycyclic group is polycyclic. One may consult \cite{DK} for
a detailed introduction to  polycyclic groups.
\medskip

The following proposition can be seen in \cite{Wo}.

\begin{prop}\label{noetherian prop}
A solvable group $G$ is polycyclic if and only if every subgroup of $G$ is finitely generated.
\end{prop}

Let $\mathbb R$ be the real line. For $a\in \mathbb R\setminus\{0\}$ and $b\in \mathbb R$,
define the affine transformation $A_{a,b}:\mathbb R\rightarrow \mathbb R$ by $A_{a,b}(x)=ax+b$ for all $x\in \mathbb R$.
Set ${\rm Aff}(\mathbb R)=\{A_{a,b}:a\in \mathbb R\setminus\{0\}, b\in \mathbb R\}$. Then  ${\rm Aff}(\mathbb R)$
is a solvable group and is called the {\it affine transformation group} on $\mathbb R$. It is known that
 ${\rm Aff}(\mathbb R)$ consists of the homeomorphisms $f$ on $\mathbb R$ satisfying that
 $|f(x)-f(y)|=c|x-y|$ for some $c=c(f)>0$ and for all $x, y\in \mathbb R$.

\begin{lem}\label{noncyclic}
Let $G$ be a subgroup of ${\rm Aff}(\mathbb R)$. If $G$ contains an $A_{a, c}$ with $|a|\not=1$, $c\in\mathbb{R}$  and an $A_{1, b}$ with $b\not=0$,
then $G$ cannot be polycyclic.
\end{lem}

\begin{proof}
WLOG, we may assume that $G$ contains an element of the form $A_{a, 0}$; otherwise, we need only
consider a conjugation of $G$  by the translation $L:=A_{1,\frac{c}{a-1}}$ on $\mathbb R$ ($LA_{a,c}L^{-1}=A_{a,0}$ and $LA_{1,b}L^{-1}=A_{1,b}$). 
It is easy to check that $A_{a,0}A_{1, b}A_{a, 0}^{-1}=A_{1, ab}$. From this we see that the set
$S:=\{A_{1, a^mb}: m\in\mathbb Z\}$ is contained in $G$. Since the subgroup $\langle S\rangle$
generated by $S$ is not finitely generated, $G$ is not polycyclic by Proposition \ref{noetherian prop}.
\end{proof}

Let $X$ be a topological space and let $G$ be a group acting on $X$. A Borel measure $\mu$ on  $X$ is called a {\it Radon measure}
if it is finite on every compact subset of $X$; it is called {\it quasi-invariant}
if for every $g\in G$ there is some $c(g)>0$ such that $\mu(g^{-1}A)=c(g)\mu(A)$ for every Borel subset $A$ in $X$;
it is called {\it invariant} if $c(g)=1$ for every $g\in G$. Clearly, if $\mu$ is quasi-invariant,
then $c(g_1g_2)=c(g_1)c(g_2)$ for every $g_1, g_2\in G$.
\medskip

The following theorem is due to Plante \cite{Pl}.

\begin{thm}\label{quai invariant}
Let $G$ be a polycyclic group acting continuously on the real line $\mathbb R$. Then there is a nontrivial $G$-quasi-invariant Radon
measure $\mu$ on $\mathbb R$. (Here, ``nontrivial" means $\mu(A)>0$ for some Borel set $A$.)
\end{thm}

Two actions $(X, G, \phi)$ and $(Y, G, \psi)$ are said
to be {\it topologically conjugate} if there is a homeomorphism $h:X\rightarrow Y$ such that
$h(\phi(g)(x))=\psi(g)(h(x))$ for every $x\in X$ and $g\in G$.
The following proposition clarifies the structure of minimal actions on $\mathbb R$ by polycyclic groups.

\begin{prop}\label{isometric action}
Let $G$ be a polycyclic group and let $\phi:G\rightarrow {\rm Homeo}(\mathbb R)$ be a continuous action. If
$(\mathbb R, G, \phi)$ is minimal, then $(\mathbb R, G, \phi)$ is topologically conjugate to an action
$(\mathbb R, G, \psi)$ with each element of $\psi(G)$ being isometric.
\end{prop}

\begin{proof}
From Theorem \ref{quai invariant}, we can take a nontrivial $G$-quasi-invariant Radon measure $\mu$ on $\mathbb R$.
Now we define a map $h:\mathbb R\rightarrow \mathbb R$ as does in \cite{Pl}:
\[h(x)=\begin{cases}\mu([0, x]), & x\geq 0;\\ -\mu([x, 0]),&  x<0. \end{cases}\]
Since $(\mathbb R, G, \phi)$ is minimal, the support ${\rm supp}(\mu)=\mathbb R$. Furthermore,  
$\mu$ contains no atoms by the minimality of $G$ and quasi-invariance of $\mu$.  These imply that $h$ is a homeomorphism. From the quasi-invariant of $\mu$, we see
that for each $g\in G$, $hgh^{-1}$ is an affine transformation on $\mathbb R$. Now define a action
$\psi:G\rightarrow {\rm Homeo}(\mathbb R)$ by $\psi(g)=hgh^{-1}$ for each $g\in G$.
\medskip

By the definition of polycyclic group,  there is a decreasing series of normal subgroups
$G=N_0\rhd N_1\rhd \cdots\rhd N_n\rhd N_{n+1}=\{e\}$ such that $N_i/N_{i+1}$ is cyclic for each $i\geq 0$.
Take $g_i\in N_i\setminus N_{i+1}$ such that $N_i/N_{i+1}=\langle g_iN_{i+1}\rangle$.
Thus for each $i\geq 0$, $N_i=\langle g_i, g_{i+1}, \cdots, g_n\rangle$.
Let $f_i=\psi(g_i)=A_{a_i, b_i}$ for each $i\geq 0$ and for some $a_i, b_i\in\mathbb R$.
\medskip

Assume to the contrary that  there is some $i_0\in\{0, 1, \cdots, n\}$ such that $f_{i_0}$ is  not isometric and for each
$i_0<i\leq n$, $f_i$ is isometric; that is $|a_{i_0}|\not=1$ and $|a_i|=1$ for $i>i_0$. Since $\psi(G)$ is polycyclic,
from Lemma \ref{noncyclic}, either $f_i=A_{-1, b_i}$ or $f_i=A_{1, 0}$ for each $i>i_0$.
\medskip

{\bf Case 1.} There is $i_1>i_0$ such that $f_{i_1}=A_{-1, b_{i_1}}$ and $f_i=A_{1, 0}$ for all $i>i_1$. Then $f_{i_1}$
has a unique fixed point $c=b_{i_1}/2$, which is also a fixed point of $f_i$ with $i<i_1$ by the normality of $N_{i_1}$.
Thus $c$ is a fixed point of $\psi(G)$. This contradicts the minimality of the action $(\mathbb R, G, \psi)$.
\medskip

{\bf Case 2.} For each $i>i_0$, $f_i=A_{1, 0}$. Thus the unique fixed point $x=b_{i_0}/(1-a_{i_0})$ of $f_{i_0}$
is also a fixed point of $f_i$ with $i<i_0$ by the normality of $N_{i_0}$. Thus $x$ is the fixed point of
$\psi(G)$, which contradicts the minimality of $(\mathbb R, G, \psi)$ again.
\medskip

So the assumption is false and each $f_i$ is isometric. Hence each element of $\psi(G)$ is isometric.
\end{proof}

\section{Proof of the main theorem}

In this section, we will prove the main theorem. We first establish the nonexistence of expansive polycyclic group
actions on the interval $[0, 1]$ in the following proposition.

\begin{prop}\label{interval case}
The closed interval $[0,1]$ admits no expansive polycyclic group actions.
\end{prop}

\begin{proof}
Assume to the contrary that there is an expansive action $([0, 1], G, \phi)$ by a polycyclic group $G$.
 By Proposition \ref{minimal open interval}, there is an open interval
 $(a, b)$ in $[0, 1]$ and a subgroup $H$ of $G$ such that $(a, b)$ is $H$-invariant and the restriction action $((a, b), H|_{(a, b)})$
is minimal and $([a, b], H|_{[a, b]})$ is expansive. Let $c>0$
be an expansivity constant for the action $([a, b], H|_{[a, b]})$. Notice that $H$ is still polycyclic. Applying Proposition \ref{isometric action}, there is an
orientation-preserving homeomorphism
$\xi:(0,1)\rightarrow \mathbb R$ such that for every $g\in H$, $\xi \phi(g)\xi^{-1}$ is an isometric transformation on $\mathbb R$.
Take $a<a'<b'<b$ such that $a'-a<c$ and $b-b'<c$. Let $A'=\xi(a')$ and $B'=\xi(b')$. By uniform continuity, there is $\delta>0$, such that
for every $x, y\in [A'-1, B'+1]$ with $|x-y|<\delta$, we always have $|\xi^{-1}(x)-\xi^{-1}(y)|<c$. Take $u ,v\in \mathbb R$ with
$|u-v|<\min\{1, \delta\}$. Then for every $g\in H$, by the isometry of $\xi \phi(g)\xi^{-1}$, we have
$$|\phi(g)(\xi^{-1}(u))-\phi(g)(\xi^{-1}(v))|$$
$$=|\xi^{-1}(\xi\phi(g)\xi^{-1})(u)-\xi^{-1}(\xi\phi(g)\xi^{-1})(v)|
=|\xi^{-1}(u)-\xi^{-1}(v)|<c.$$
This contradicts the expansivity of $([a, b], H|_{[a, b]})$.
\end{proof}

The following proposition can be seen in \cite{Na}.
\begin{prop}\label{minimal sets}
Let a group $G$ act continuously on the circle $\mathbb S^1$ and let $\Lambda$ be a minimal closed subset of $(\mathbb S^1, G)$.
Then there are three case: (a) $\Lambda=\mathbb S^1$; (b) $\Lambda$ is a Cantor set; (c) $\Lambda$ is finite.
\end{prop}

\begin{proof}[Proof of Theorem \ref{mainthm}]
Assume to the contrary that there is an expansive action on $\mathbb S^1$ by a polycyclic group $G$.
Let $\Lambda$ be a minimal closed set for the action. From Proposition \ref{minimal sets}, we discuss into three cases.
\medskip

{\bf Case 1.} $\Lambda=\mathbb S^1$. By the amenability of polycyclic groups, there is a $G$-invariant probability
Borel measure $\mu$ on $\mathbb S^1$. Since $(\mathbb S^1, G)$ is minimal, the support ${\rm supp}\mu=\mathbb S^1$ and $\mu$
has no atoms. Thus similar to the construction of $h$ in Proposition \ref{isometric action}, we see that
 $(\mathbb S^1, G)$ is topologically conjugate to an isometric action of $G$ on $\mathbb S^1$. So, $(\mathbb S^1, G)$
is not expansive.  This is a contradiction.
\medskip

{\bf Case 2.}  $\Lambda$ is a Cantor set. Take a maximal interval $(\alpha, \beta)$ in $\mathbb S^1\setminus \Lambda$.
 Let $H=\{g\in G: g(\alpha, \beta)=(\alpha, \beta)\}$. Then $H$ is a subgroup of $G$ with index  $[G:H]=\infty$.
 Similar to the proof of Claim C in Proposition \ref{minimal open interval}, we see that $([\alpha, \beta], H|_{[\alpha, \beta]})$
 is expansive. This contradicts Proposition \ref{interval case}.
 \medskip

 {\bf Case 3.} $\Lambda$ is finite. Fix a maximal interval $(\alpha, \beta)$ in $\mathbb S^1\setminus \Lambda$.
 Let $H=\{g\in G: g(\alpha, \beta)=(\alpha, \beta)\}$. Then $[G:H]<\infty$.
 If the number of $\Lambda$ is greater than $1$, then $([\alpha, \beta], H|_{[\alpha, \beta]})$ is expansive
 by Lemma \ref{finite index expansive}. This contradicts Proposition \ref{interval case}. If
 $\Lambda$ consists of only one point, say $O$. We view $\mathbb S^1$ as the quotient of $[0, 1]$
 by collapsing the two endpoints $\{0, 1\}$ to one point $O$. Then the action of $H$ on $\mathbb S^1$
 can be naturally lifted to an action on $[0, 1]$. Clearly, the lifting action is still
 expansive. This contradicts Proposition \ref{interval case} again.

  \medskip
 All together, we see that the assumption is false and thus complete the proof.
 \end{proof}


\end{document}